\documentclass[12pt]{amsart}

\usepackage{amsthm,amsmath,amssymb,graphics,graphicx,hyperref,epstopdf}

\title[Bounding Projective Dimension]{Bounding the Projective Dimension of a Squarefree Monomial Ideal via Domination in Clutters}

\author{Hailong Dao and Jay Schweig}
\newtheorem{theorem}{Theorem}[section]
\newtheorem{proposition}[theorem]{Proposition}
\newtheorem{corollary}[theorem]{Corollary}
\newtheorem{lemma}[theorem]{Lemma}

\newtheorem{observation}[theorem]{Observation}                                               

\theoremstyle{definition}
\newtheorem{definition}[theorem]{Definition}
\newtheorem{remark}[theorem]{Remark}

\newcommand{\flo}[2]{\big\lfloor \frac{#1}{#2} \big\rfloor}

\newcommand{\cei}[2]{\big\lceil \frac{#1}{#2} \big\rceil}
\newcommand{\Bflo}[2]{\Big\lfloor \frac{#1}{#2} \Big\rfloor}

\newcommand{\Bcei}[2]{\Big\lceil \frac{#1}{#2} \Big\rceil}

\newcommand{\C}{\mathcal{C}}          
\DeclareMathOperator{\K}{k}

\DeclareMathOperator{\bh}{BigHeight}
\DeclareMathOperator{\pd}{pd}
\DeclareMathOperator{\cd}{cd}
\DeclareMathOperator{\reg}{reg}

\DeclareMathOperator{\lcm}{lcm}

\DeclareMathOperator{\is}{is}

\begin{document} 

\maketitle

\begin{abstract}
We introduce the concept of edgewise domination in clutters, and use it to provide an upper bound for the projective dimension of any squarefree monomial ideal.  We then compare this bound to a bound given by Faltings.  Finally, we study a family of clutters associated to graphs and compute domination parameters for certain classes of these clutters. 
\end{abstract}

\section{Introduction}

Fix a field $\K$, and let $S = \K[x_1, x_2, \ldots, x_n]$. Let $I \subset S$ be a homogenous ideal in $S$. Two fundamental invariants of $I$ are its projective dimension $\pd(I)$ and its Castelnuovo-Mumford regularity $\reg(I)$. 
There has been tremendous interest in understanding these numbers, due to their connections to topics in algebraic geometry, commutative algebra, and combinatorics. 

One of the main results in  this paper  provides a new upper bound on these invariants when $I$ is a square-free monomial ideal. Even in this special case, the problem is quite subtle and significant, as illustrated by its well-known link to combinatorial topology.   Let  $\Delta$ be a simplicial complex on vertex set $[n] = \{1, 2, \ldots, n\}$. In general, the minimal non-faces of any simplicial complex $\Delta$ form a \emph{clutter} on $[n]$, which is a collection of subsets of $[n]$, none of which contains another.  If $\C$ is a clutter, we let $I(\C)$ be the ideal generated by monomials corresponding to the sets in $\C$ (so that $I(\C)$ is an edge ideal when the sets in $\C$ are edges of a graph).  In this way, every squarefree monomial ideal arises as the Stanley-Reisner ring of some simplicial complex $\Delta$.  Through Hochster's Formula, bounds on $\pd(I(\C))$ yield results about the homology of $\Delta$. 

Let us now describe our bound in more detail. Let $I\subset S$ be a square-free monomial ideal; we impose no other restrictions on $I$. Obviously, one has $\pd(S/I)\leq n$ by the Auslander-Buchsbaum formula. The only nontrivial bound that we are aware of follows from a quite general result due to Faltings (see Theorem \ref{faltings}). Our bound is defined combinatorially and results from a new clutter domination parameter which we believe to be of independent interest. Recall that when the minimal non-faces of a complex $\Delta$   all have cardinality two (that is, $\Delta$ is a \emph{flag} complex), then the Stanley-Reisner ideal of $\Delta$ is the edge ideal of the graph with vertex set $[n]$ and edge set $\{(i, j) : \{i, j\}$ is not a face of $\Delta\}$.  In \cite{meandlong}, we employ domination parameters and an induction argument to bound the projective dimension of this ideal.  In particular, we introduce a new graph domination parameter called \emph{edgewise domination}, which works especially well in the bounding of projective dimension.   
In this paper we generalize edgewise domination from graphs to clutters, and show that if $\epsilon(\C)$ is this new domination parameter of a clutter $\C$ (see Section \ref{secclut}) and $V(\C)$ is the set of vertices of $\C$, then the following holds.\\

\noindent
\textbf{Theorem \ref{edgewise}.}  \emph{For any clutter $\C$, $\pd(S/I(\C)) \leq |V(\C)| - \epsilon(\C)$}.  \\

Just as in \cite{meandlong}, bounds on the projective dimension of a squarefree monomial ideal $I$ give rise to homological bounds on the complex $\Delta$ whose Stanley-Reisner ideal is $I$.

Another class of clutters considered in our paper is the family  $\{\C_k(G) : 2 \leq k \leq |V(G)|\}$ of clutters associated to a graph $G$.  These clutters (which correspond to Stanley-Reisner ideals of complexes studied by Szab\'o and Tardos in \cite{st}, and which were later examined by Dochtermann and Engstr\"om in \cite{de}) generalize edge ideals, as $\C_2(G)$ is simply the set of edges of $G$.  One of our motivations for the study of these clutters is Theorem \ref{simplicial}, which states that a vertex of $G$ is simplicial if and only if it is simplicial in $\C_k(G)$ for all $k \leq |V(G)|$.   

Edge ideals have been generalized to \emph{path ideals} (introduced in \cite{pathidealsintro} and studied further in many papers such as \cite{faridi,  pathideals1, cgmpp, adampathideal}).  The clutters $\C_k(G)$ thus provide another generalization of edge ideals which, unlike path ideals, do not require an orientation of the graph $G$.  


The class $\C_k(G)$ of clutters raises many other interesting questions.  We begin a study of the domination parameters of these clutters, obtaining formulas for the projective dimension of $\C_k(G)$ when $G$ is a path and bounds for the projective dimension of $\C_k(G)$ when $G$ is a cycle, see Section \ref{seccon}.  

We begin the paper with Section \ref{secpre}, in which we list some useful preliminaries. Section \ref{secclut}  introduces the edgewise domination parameter for clutters. Here one of our main technical results, Theorem \ref{edgewise}, is proved. Section \ref{seccon} discusses \emph{connected graph ideals}, which generalize edge ideals. We also compute domination parameters for these clutters in the case when the associated graph is a path or a cycle, and compare them with the projective dimensions of these clutters (Theorem \ref{pathcycle}).  Section \ref{secfalt} recalls a consequence of a result by Faltings which gives a simple upper bound for the projective dimensions of clutters. For completeness, we give a short proof of Faltings' bound in the case when $I$ is a squarefree monomial ideal, and we compare this bound to the ones obtained in this paper. Finally, in Section \ref{apps}, we derive corollaries of our results which bound the homology of simplicial complexes.

\section{Preliminaries}\label{secpre}

Fix a field $\text{k}$, and let $S = \text{k}[x_1, x_2, \ldots, x_n]$.  If $I \subseteq S$ is a squarefree monomial ideal, we can identify its unique minimal set of generators with a \emph{clutter}, defined as follows.

\begin{definition}
Let $V$ be a finite set.  A \emph{clutter} $\C$ with vertex set $V$ consists of a set of subsets of $V$, called the \emph{edges} of $\C$, with the property that no edge contains another.
\end{definition}

We write $V(\C)$ to denote the vertices of $\C$, and $E(\C)$ to denote its edges.  For simplicity, we often identify vertex sets with subsets of $\mathbb{N}$.  We say two vertices of $\C$ are \emph{neighbors} if there is an edge of $\C$ which contains both these vertices.  We also let $\is(\C)$ denote the set of vertices of $\C$ not appearing in any edge, and we write $\overline{\C}$ to denote $\C$ with its isolated vertices removed.  If $A\subseteq V(\C)$ contains no edge of $\C$, we say $A$ is \emph{independent}.

\begin{definition}
Let $\C$ be a clutter.  Key to our constructions will be two clutters obtained from $\C$, defined as follows.  

\begin{itemize}
\item If $|A| \geq 1$, $\C + A$ is the clutter whose edges are the minimal sets of $E(\C) \cup \{A\}$ and vertex set $ V(\C)$. 

\item $\C : A$ is the clutter whose edges are the minimal sets of $\{e  \setminus A: e\in E(\C)\}$ and whose vertex set is $V(\C) \setminus A$.  
\end{itemize}
\end{definition}

Note that some clutters may have edges which contain only one element.  We call such edges \emph{trivial}.  Trivial edges are often easy to handle, since if $\{v\}$ is a trivial edge of some clutter $\C$, then it is the only edge in which the vertex $v$ appears.  However, it is important to distinguish between a \emph{vertex in a trivial edge} and an \emph{isolated vertex}. 

If $A \subseteq \{1, 2, \ldots, n\}$, we write $x^A$ to denote the squarefree monomial $\prod_{i\in A}x_i$.  If $\C$ is a clutter, the associated monomial ideal $I(\C)$ is given by
\[
I(\C) = ( x^e : e\in E(\C)).
\]

The operations introduced above correspond to standard operations on ideals as follows.  Throughout, let $\C$ be a clutter.

\begin{observation}\label{opsob}
Let $A \subseteq V(\C)$.  Then 
\[
(I(\C), x^A) = I(\C + A)\text{ and } I(\C): x^A = I(\C : A). 
\]
\end{observation}

This observation allows us to prove a lemma essential to an inductive bounding of a clutter's projective dimension.  

\begin{lemma}\label{exact}
Let $A \subseteq V(\C)$.  Then 
\[
\pd(\C) \leq \max\{ \pd(\C + A), \pd(\C :A)\},
\]
where here (and throughout), we write $\pd(\C)$ to mean $\pd(S / I(C))$, the projective dimension of the quotient $S / I(\C)$.  
\end{lemma}

\begin{proof}
This follows from Observation \ref{opsob} and the natural short exact sequence 
\[
0 \rightarrow \frac{S}{I(\C) : x^A} \rightarrow \frac{S}{I(\C)} \rightarrow \frac{S}{(I(\C), x^A)}\rightarrow 0,
\]
which gives $\pd(I(\C)) \leq \max\{\pd(I(\C): x^A), \pd((I(\C), x^A))\}$.  The final statement is immediate.
\end{proof}

\begin{definition} 
Let $\Phi$ be a collection of clutters.  We say $\Phi$ is \emph{hereditary} if both $\C + A$ and $\C : A$ are in $\Phi$ for any $A \subseteq V(\C)$, as is $\overline{\C}$.
\end{definition}

\section{Domination in Clutters}\label{secclut}

Let $\C$ be a clutter.  Here we generalize the concept of \emph{edgewise domination}, first introduced in \cite{meandlong}, to clutters:

\begin{definition}\label{clutdef}
 We call $F \subseteq E(\C)$ \emph{edgewise dominant} if every vertex in $V(\overline{\C})$ not contained in some edge of $F$ or contained in a trivial edge has a neighbor contained in some edge of $F$.  We define $\epsilon(\C)$ by $\epsilon(\C) = \min \{|F| : F \subseteq E(\C)$ is edgewise dominant$\}$.

\end{definition}

Our main result from this section is the following.

\begin{theorem}\label{edgewise}
For any clutter $\C$, $\pd(\C) \leq |V(\C)| - \epsilon(\C)$.  
\end{theorem}


We first need the following lemma, which generalizes Theorem 3.1 from \cite{meandlong}.

\begin{lemma}\label{meta}
Let $\Phi$ be a hereditary class of clutters, and let $f:\Phi \rightarrow \mathbb{N}$ be a function such that $f(\overline{\C}) = f(\C)$ for all clutters $\C$ 
in $\Phi$, $f(\C) \leq |V(\C)|$ when $\C$ has no edges, and $f(\C) = 0$ when $\C$ has only trivial edges.  Further suppose that $f$ satisfies the following: for any $\C \in \Phi$ with at least one non-trivial edge, there exists a sequence of sets $A_1, A_2, \ldots, A_t$ such that, writing $\C_i$ for the clutter $ \C + A_1 + A_2 + \cdots + A_i$, the following two properties are satisfied. 

\begin{itemize}
\item $|\is(\C_t)| > 0$ and $f(\C_t) + |\is(\C_t)| \geq f(\C)$, and
\item For each $i$, $f(\C_{i-1} : A_i) + |\is(\C_{i-1})| + |A_i| \geq f(\C)$.
\end{itemize}

Then for any $\C \in \Phi$, we have 
\[
\pd(\C) \leq |V(\overline{\C})| - f(\C). 
\]
\end{lemma}

\begin{proof}
Note that it suffices to prove the statement for clutters without isolated vertices.  Indeed, if $\C$ had isolated vertices and we prove the lemma for $\overline{\C}$, we would then have $\pd(\C) = \pd(\overline{\C}) \leq |V(\overline{\C})| - f(\overline{\C}) = |V(\overline{\C})| - f(\C)$.  Similarly, if $\C$ had only trivial edges, we would have $\pd(\C) \leq |V(\C)| - f(\C) = |V(\C)|$, which holds.

Now suppose $\C \in \Phi$ has no isolated vertices and at least one non-trivial edge.  We use a repeated application of Lemma \ref{exact}.  There are two cases to consider.  In both cases, we induct on the number of vertices of $\C$ (the base case of one vertex being immediate).  First suppose $\pd(\C) \leq \pd(\C_1) \leq \pd(\C_2) \leq \cdots \leq \pd(\C_t) = \pd(\overline{\C_t})$.  As $|\is(\C_t)| > 0$, by induction we have $\pd(\overline{\C_t}) \leq |V(\overline{\C_t})| - f(\C_t) \leq |V(\C)|- |\is(\C_t)| - f(\C) + |\is(\C_t)| = |V(\C)| - f(\C)$.

Next, suppose that $\pd(\C) \leq \pd(\C_1) \leq \cdots \leq \pd(\C_{i-1}) \leq \pd(\C_{i-1}:A_i)$ for some $i\geq 1$ (where $\C_0 = \C$).  As $\C_{i-1}:A_i$ has fewer vertices than $\C$, by induction we have 
\begin{align*}
\pd(\C) &\leq \pd(\C_{i-1}: A_i) \leq |V(\overline{\C_{i-1} : A_i})| - f(\C_{i-1} : A_i) \\
&\leq |V(\overline{\C_{i-1}})| - |A_i| - f(\C_{i-1} : A_i) \\
&\leq |V(\C)|- |\is(\C_{i-1})| - |A_i| + |\is(\C_{i-1})| + |A_i| - f(\C) \\
&=|V(\C)| - f(\C).  \qedhere
\end{align*}  
\end{proof}

\begin{proof}[Proof of Theorem \ref{edgewise}]

Let $\C$ be a clutter.  We use Lemma \ref{meta}, letting $\Phi$ be the class of all clutters, and thus can assume $\C$ has no isolated vertices.  We can also assume that $\C$ has no trivial edges.  Indeed, if $\C$ had $k$ trivial edges and $\C'$ denoted $\C$ with these edges and vertices removed, proving the bound $\pd(\C') \leq |V(\C')| - \epsilon(\C)$ would also prove the corresponding bound for $\C$, as we would have $\pd(\C) = \pd(\C') + k$ and $|V(\C)| = |V(\C')| + k$.

Let $x$ be a vertex contained in some non-trivial edge of $\C$, and let $y_1, y_2, \ldots, y_t$ be the neighbors of $x$.  We apply Lemma \ref{meta} with $A_i = y_i$ for each $i$ and $f(\C) = \epsilon(\C)$.  Then $x$ is isolated in $\C_t$, by construction.  Let $F \subseteq E(\C_t)$ be an edgewise-dominant set of $\C_t$ with $|F| = \epsilon(\C_t)$.  Write $\is(\C_t) = \{x, z_1, z_2, \ldots, z_k\}$, pick an edge $e\in E(\C)$ containing $x$ and, for each $z_i$, pick an edge $e_i \in E(\C)$ containing $z_i$.  

Note that every edge of $F$ is an edge of $\C$, and let $F' = F \cup \{e, e_1, e_2, \ldots, e_k\}$.  We claim $F'$ is an edgewise-dominant set of $\C$.  Indeed, each $y_i$ is a neighbor of $x$, which is contained in $e$.  If $v$ is any other non-isolated vertex of $\C$, then either $v \in \overline{\C_t}$, in which case it is contained in an edge of $F$ or has a neighbor that is, or $ v = z_i$ for some $i$, in which case it is contained in $e_i$.  

Thus, $|F'| \leq  |F| + |\is(\C_t)| = \epsilon(\C_t) + |\is(\C_t)| \geq \epsilon(\C)$, meaning $\epsilon(\C)$ satisfies the first condition of Lemma \ref{meta}.  

The second condition is verified similarly: pick $i$, let $F \subseteq E(\C_{i-1} : y_i)$ be an edgewise-dominant set with $|F| = \epsilon(\C_{i-1}: y_i)$, and let $e$ be an edge of $\C$ that contains both $x$ and $y_i$.  If $f \in F$, then either $ f \cup \{y_i\}$ or $f$ must be an edge of $\C$ (whichever is, call this edge $f'$).  For each $z \in \is(\C_{i-1}: y_i)$, pick an edge $e_z \in \C$ containing it.  Set $F' = \{f ' : f\in F\} \cup\{e_z : z\in \is(\C_{i-1}: y_i)\} \cup \{e\}$.  We claim $F'$ is an edgewise-dominant set of $\C$.  Indeed, each $y_j$ is a neighbor of $x$, which is contained in the edge $e$.  If $v \neq y_j$ is a non-isolated vertex of $\C_{i-1}: y_i$, then it is either contained in some edge $f \in F$ or some neighbor of it is, and this property carries over to $\C$, replacing $f$ with $f'$.  Thus, 
$|F'| \leq |F| + |\is(\C_{i-1} : y_i)| + 1 = \epsilon(\C_{i-1} : y_i) + |\is(\C_{i-1}: y_i)| + |\{y_i\}| \geq \epsilon(\C)$. \end{proof}

\begin{remark}
For a clutter $\C$, let $i(\C)$ denote the smallest cardinality of an independent set of $\C$.  It is straightforward to show that $|V(\C)| - i(\C) = \bh(I(\C))$.  As $\pd(\C) \geq \bh(I(\C))$, this observation and Theorem \ref{edgewise} give: 
\[
|V(\C)| - i(\C) \leq \pd(\C) \leq |V(\C)| - \epsilon(\C).  
\]
\end{remark}

\section{Connected Graph Clutters and Ideals}\label{seccon}

Recall that the \emph{edge ideal} of a simple graph $G$ on vertex set $[n]$ is just $I(\C)$ where $\C$ is the clutter of edges of $G$.  That is, 
\[
I(G) = ( x_ix_j : (i, j) \in E(G) ) .
\]

Edge ideals have been generalized to \emph{path ideals}; here we consider the equivalent formulation in terms of clutters.  For a directed graph $G$, let $P_k(G)$ denote the clutter whose faces are all $k$-vertex subsets of $V(G)$ which form a directed path in $G$.  Equivalently, a $k$-subset $A = \{v_1, v_2, \ldots, v_k\}$ is an edge of $P_k(G)$ if and only if there exists a permutation $\sigma$ on $k$ letters such that $v_{\sigma(i)} \rightarrow  v_{\sigma(i+1)}$ is a directed edge of $G$ for any $i$ with $1\leq i \leq k-1$.  Path ideals are then the ideals corresponding to the clutters $P_k(G)$.  

\begin{definition}
In \cite{russ}, Woodroofe defines a vertex $v$ of a clutter $\C$ to be \emph{simplicial} if, for any two distinct edges $e, f \in E(\C)$ with $v \in e \cap f$, there exists an edge $g \in E(\C)$ with $g \subseteq (e \cup f) - v$.  
\end{definition}

If $\C$ is the edges of a graph $G$, then a vertex $v$ is simplicial if and only if $(x,y)$ is an edge of $G$ whenever $x$ and $y$ are neighbors of $v$.  Unfortunately, simplicial vertices in graphs are not necessarily simplicial in the associated path ideals:  if $G$ is the graph shown in Figure \ref{path}, then every vertex of the graph save the central one is simplicial, but $P_3(G)$ has no simplicial vertex. 

\begin{figure}[htp]
\centering
\includegraphics[height = .8in]{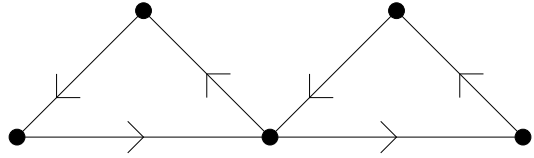}
\caption{A graph with simplicial vertices for which $\C_3(G)$ has no simplicial vertices.}\label{path}
\end{figure}

Here, we offer some evidence in favor of another generalization of edge ideals.  These clutters are the minimal non-faces of simplicial complexes studied by Szab\'o and Tardos in \cite{st} and studied further in \cite{de}.  If $X\subseteq V(G)$ we write $G[X]$ to denote the \emph{induced subgraph} on $X$, the graph with vertex set $X$ whose edges consist of all edges of $G$ whose endpoints lie in $X$.

\begin{definition}
Let $G$ be a graph, and choose $k$ with $2 \leq k \leq |V(G)|$.  The \emph{$k$-connected graph clutter} of $G$, for which we write $\C_k(G)$, is defined to be the clutter with the following edge set:
\[
E(\C_k(G)) = \{A\subseteq V(G) : G[A] \text{ is connected and } |A| = k\}.
\] 
\end{definition}

Note that the clutter $\C_k(G)$ is \emph{$k$-uniform} (meaning all its edges are of cardinality $k$) and $I(\C_2(G)) = I(G)$, as the only induced connected subgraphs of size $2$ are edges.  We believe that the ideals $I(\C_n(G))$ provide a more natural generalization of edge ideals.  One piece of evidence for this belief is the following theorem.

\begin{theorem}\label{simplicial}
A vertex $v$ of a graph $G$ is simplicial (in the graph sense) if and only if it is simplicial in $\C_k(G)$ for any $k$. 
\end{theorem} 

\begin{proof}
Note that the ``only if'' direction is immediate: If $v$ is not simplicial in $G$, then it is not simplicial in $\C_2(G)$, by definition.  

For the ``if'' direction, suppose $v$ is simplicial in $G$, and let $k \geq 2$.  If the connected component of $G$ which contains $v$ has fewer than $k$ vertices, then $v$ is simplicial, as it is not contained in any edge (and so satisfies the condition vacuously).  Otherwise, let $A$ and $B$ be two subsets of $V(G)$ of size $k$ so that $G[A]$ and $G[B]$ are connected, and $v \in A \cap B$.  Then both $A$ and $B$ must intersect the neighbors of $v$.  Let $A_v\subseteq A$ be the vertices in $A$ that are adjacent to $v$, and define $B_v \subseteq B$ similarly.  Since $v$ is simplicial in $G$, $G[A_v \cup B_v]$ is connected (in fact, complete), meaning each of $G[A -v]$ and $G[B - v]$ is connected.  Thus $G[(A \cup B ) - v]$ is connected.  As $A \neq B$ and $|A| = |B| = k$, $|(A \cup B)  - v| \geq k + 1$.  Thus there is a subset $C \subseteq (A \cup B ) - v$ such that $|C| = k$ and $G[C]$ is connected, which shows $v$ is simplicial in $\C_k(G)$. \end{proof}

By definition, no clutter that is not uniform can be a connected graph clutter.  The following shows that not all uniform clutters are connected graph clutters.  

\begin{observation}
Let $\C$ be the clutter with edges $\{a, b, c\}$ and $\{c, e, d\}$.  Then $\C$ is not a connected graph clutter.
\end{observation}

\begin{proof}
Suppose there were a graph $G$ with $\C = \C_3(G)$.  Since $\{a, b, c\} \in E(\C)$, it must be the case that at least one of $a$ or $b$ is a neighbor of $c$ (without loss, assume it is $b$).   Similarly, we can assume (again, without loss), that $d$ is a neighbor of $c$.  But then we would have $\{b, c, d\} \in E(\C)$, which is a contradiction.
\end{proof}

Next we compute domination parameters for $\C_k(G)$ when $\C$ is a path or cycle. Our results are summarized in the following theorem.\\

\begin{theorem}\label{pathcycle}
Let $P_n$ $(\Gamma_n)$ denote the path (cycle) with $n$ vertices. The following table gives the domination  numbers for the $k$-connected clutter of $P_n$ $(\Gamma_n)$:\\
\begin{center}
\begin{tabular}{|c|c|c|c|}
\hline
\text{Clutter} $\C$  & $i(\C)$ & $\epsilon(\C)$  & $\pd(\C)$  \\

\hline
$\C_k(\Gamma_n)$ & $\lceil \frac{(k-1)n}{k+1} \rceil$  & $\lceil \frac{n}{3k-2} \rceil$ & $\lfloor \frac{n}{k+1} \rfloor+ \lceil \frac{n}{k+1} \rceil$\\
\hline
$\C_k(P_n)$ & $\lceil \frac{kn}{k+1} \rceil- \lfloor \frac{n+1}{k+1} \rfloor$ &  $\lceil \frac{n}{3k-2} \rceil$ &         $\lfloor \frac{n}{k+1} \rfloor + \lfloor \frac{n+1}{k+1} \rfloor$\\
\hline
\end{tabular}\\
\end{center}
(The projective dimensions are included for completeness as they are already known; see, for instance, \cite{faridi, adampathideal}).
\end{theorem}
\vspace{3mm}
\begin{proof}
We start with $\C = \C_k(\Gamma_n)$. Consider an independent dominating set $S$ of $\C$ and let $\{x_1,\cdots, x_l\} = V(\C) \backslash S$. Let $t_i = x_{i+1}-x_i$ for $1\leq i\leq l-1$ and $t_l = x_1+n-x_l$. The fact that $S$ is independent and dominating is equivalent to the sequence $\{t_i\}$ satisfying the following properties (set $t_{l+1}=t_0$):
  
\[
\begin{cases}
1\leq t_i\leq k,\\
t_i +t_{i+1} \geq k+1, & 1 \leq i\leq l\\
\sum_{i=1}^{l} t_i = n
\end{cases}
\]

Adding all the inequalities $t_i +t_{i+1} \geq k+1$ we get $l(k+1) \leq  2n$. It follows that $l\leq \flo{2n}{k+1}$.
To show that the equality can be obtained, let $2n = l(k+1) + r$ with $0\leq r\leq k$. Set $t_{2i+1}=a= \cei{k+1}{2}$,  $t_{2i}=\flo{k+1}{2}$.
Then the conditions $t_i+t_{i+1}\geq k+1$ are satisfied automatically. Showing that the $t_i$ can be modified to satisfy the above system of inequalities is equivalent to showing:
\[a\Bcei l2 + b\Bflo l2 \leq \frac{(k+1)l+r}{2} \leq kl \]

If $l=1$ we simply pick $t_1=n$, so assume $l\geq 2$.  We can also assume $k\geq 2$. The rightmost inequality is equivalent to $l+r \leq kl$ which is easily seen to be true when $k,l\geq 2$. The leftmost one is equivalent to (taking into account that $a+b=k+1$):
$$(k+1)\left(\Bcei l2 -\frac l2\right) \leq b\left(\Bcei l2 - \Bflo l2\right) + r $$
or 
$$(k+1-2b)\left(\Bcei l2 -\frac l2\right) \leq r$$ 
If $k+1$ or $l$ is even then the left hand side is $0$. If they are both odd, then $r$ must be odd, and the left hand side is $1/2$.  Thus, the maximal value for $l$ is $\flo {2n}{k+1}$ and $i(\C) = n- \flo {2n}{k+1} = \cei{(k-1)n}{k+1}$.

To compute $\epsilon(\C)$, first we note that since each edge in $\C$ can dominate at most $k + k-1+k-1 =3k-2$ vertices, we must have $\epsilon(\C)\geq \cei{n}{3k-2}$. It is easy to see that equality can be achieved.  

Let's now consider $\C = \C(P_n)$. As before, $n-i(\C)$ is the largest value of $l$ such that the following system have integer solutions: 
\[
\begin{cases}
1\leq t_1\leq k,\\
t_i +t_{i+1} \geq k+1, & 1 \leq i\leq l-1\\
t_l+t_{l+1} \geq k, \\
\sum_{i=1}^{l+1} t_i = n
\end{cases}
\]

Let $A$ and $B$ be the set of odd and even numbers between $1$ and $l$ respectively. If $l$ is odd, then we have:
\begin{align*}
 n+1  &= (t_1+t_2) + \cdots + (t_l+t_{l+1}+1) \\
         & \geq |A|(k+1)
\end{align*}
It follows that $|A|\leq \flo{n+1}{k+1}$. It is clear  that $|B|(k+1) \leq n$, so  $|B|\leq \flo{n}{k+1}$. Thus 
$$l = |A|+ |B| \leq  \flo{n}{k+1} + \flo{n+1}{k+1}$$
When $l$ is even, we have $|A|, |B| \leq \flo{n}{k+1}$, so the same inequality holds. That $l$ can achieve the upper bound  can be proved similarly to the case of $\C(\Gamma_n)$ and is left as an exercise for the reader. 

The computation of $\epsilon(\C(\Gamma_n))$ is identical to the case of paths. 
\end{proof}

\section{Faltings' Bound on Projective Dimensions}\label{secfalt}

Write $\cd(I)$ to mean the cohomological dimension of an ideal $I$.  As $\cd(I) = \pd(S/I)$ for any squarefree monomial ideal (see \cite[Corollary 4.2]{sw}), the following general bound of Faltings (\cite{faltings}) applies to bound $\pd(\C)$ for any clutter $\C$.  

\begin{theorem}[\cite{faltings}]\label{faltings}
For any ideal $I \subset S = \text{\emph{k}}[x_1, x_2, \ldots, x_n]$, 
\[
\cd(I) \leq n - \left\lfloor \frac{n-1}{\bh(I)}\right\rfloor.
\]
\end{theorem}

The next proposition follows directly from the Taylor Resolution (see \cite{adamandtai}).  

\begin{proposition}\label{taylor}
Let $I \subseteq S$ be a monomial ideal generated by monomials $m_1, m_2, \ldots, m_t$.  For any set $A \subseteq \{1, 2, \ldots, t\}$, we write $\deg(A)$ to denote the quantity $\deg\lcm\{ m_i : i \in A\}$.  Then 
\[
\reg(I) \leq \max\{ \deg(A) - |A| : A \subseteq \{1, 2, \ldots, t\} \} + 1
\]
\end{proposition}

Using Proposition \ref{taylor} and Alexander duality, we can easily recover Faltings' bound for the projective dimension of a squarefree monomial ideal.  

\begin{proof}[Proof of Theorem \ref{faltings}]
Let $I$ be a squarefree monomial ideal, and $I^\vee$ its Alexander dual.  Then $I^\vee$ is also squarefree, so we can write $I^\vee = I(\C)$ for some clutter $\C$ without isolated vertices.  Then $\pd(S/I) = \reg(I(\C))$, so we use Proposition \ref{taylor} to bound this quantity.  Let $E(\C) = \{e_1, e_2, \ldots, e_t\}$ be the edge set of $\C$.  We claim that $\deg(A) - |A|$ is maximized when $\deg(A) = n$.  Indeed, $\lcm\{e_i : i \in A\}$ is squarefree for any $A$.  If this lcm does not contain $x_i$ for some $i$, let $e_j$ be an edge of $\C$ containing $x_i$.  Then $\deg(A \cup \{j\}) \geq \deg(A) + 1$, so $\deg(A \cup \{j\}) - |A \cup \{j\}|\geq \deg(A) -|A|$, meaning we may consider $A \cup \{j\}$ rather than $A$ as a set maximizing $\deg(A) - |A|$.  

Thus, $\reg(I(\C)) \leq n - \alpha + 1$, where $\alpha$ is the smallest cardinality of a set $F\subseteq E(\C)$ such that every vertex of $\C$ is contained in some edge of $F$.  Now let $d$ be the maximal cardinality of an edge of $\C$.  Then, since each edge contains at most $d$ vertices, we have $d \alpha \geq n$, and so $\alpha \geq n/d$.  Thus, $\reg(I(\C)) \leq n - \alpha + 1 \leq n - n/d + 1$.  Since $\reg(I(\C))$ is an integer, we have $\reg(I(\C)) \leq n - \lceil n/d \rceil + 1$. 

Now note that $\lceil n/d\rceil -1 \geq \lfloor (n-1)/d \rfloor$.  Since $d$ is the maximal degree of a generator of $I(\C) = I^\vee$, we have $d = \bh(I)$.  Thus, 
\[
\pd(S/I) = \reg(I(\C)) \leq n - \lceil n/d \rceil + 1 \leq n - \left\lfloor  \frac{n-1}{d} \right\rfloor = n - \left\lfloor \frac{n-1}{\bh(I)}\right\rfloor.
\]
\end{proof}

\begin{proposition}
Let $\C$ be a clutter.  If $\epsilon(\C) \geq i(\C)/(|V(\C)| - i(\C))$, Theorem \ref{edgewise} improves upon (or recovers) the bound of Theorem \ref{faltings}.
\end{proposition}

\begin{proof}
For readability, write $\epsilon, i,$ and $v$ in place of $\epsilon(\C), i(\C),$ and $|V(\C)|$, respectively.  If $\epsilon \geq i /(v - i)$, then 
\begin{align*}
\epsilon &\geq \frac{i}{v-i} \Rightarrow (v-i)\epsilon \geq i \Rightarrow v\epsilon - i \epsilon - i \geq 0 \Rightarrow  v\epsilon - i \epsilon - i + v \geq v \Rightarrow (\epsilon + 1)(v - i) \geq v \\
&\Rightarrow (\epsilon + 1)(v - i) > v-1 \Rightarrow \epsilon + 1 > \frac{v-1}{v-i} \Rightarrow \epsilon \geq \left\lfloor \frac{v-1}{v-i} \right\rfloor
\end{align*}
Setting $|V(\C)| = n$, the above translates to $\epsilon(\C) \geq \lfloor (n-1)/(n-i(\C)) \rfloor$.  Since $n - i(\C) = \bh(I(\C))$, the result follows.  
\end{proof}

\begin{remark}
If $\C$ is a clutter, write $\reg(\C)$ to denote the (Castelnuovo-Mumford) regularity of $I(\C)$.  As $\reg(\C) = \pd(\C^\vee)$, our bound from Theorem \ref{edgewise} gives us that 
\[
\reg(\C) \leq |V(\C)| - \epsilon(\C^\vee).
\]
In \cite{lm}, the authors use an interesting new technique to bound $\reg(\C)$.  If we let $\C$ be the edges of the pentagon then $\epsilon(\C^\vee) = 1$, and the method in \cite{lm} gives $\reg(\C)   \leq 4 = |V(\C)| - \epsilon(\C^\vee)$, meaning the two bounds are equal in that case.  However, the method also gives $\reg(\C^\vee) \leq 4$, whereas our method gives $\reg(\C^\vee) \leq |V(\C)| - \epsilon((\C ^\vee)^\vee) = |V(\C)| - \epsilon(\C) = 5 - 2  = 3$.  So far, we have been unable to find a clutter for which the bound given in \cite{lm} is stronger than the bound associated to Theorem \ref{edgewise}.  
\end{remark}

\section{Homological Consequences}\label{apps}

We begin this section by stating the well-known \emph{Hochster's Formula} in the language of clutters.  First, we need the following definition.  

\begin{definition}
Let $\C$ be a clutter without isolated vertices.  We write $\Delta_\C$ to denote the \emph{Stanley-Reisner Complex} of $\C$, the simplicial complex on $V(\C)$ whose faces are independent sets.  For $X\subseteq V(\C)$, we write $\Delta_\C[X]$ to denote the subcomplex of $\Delta_\C$ consisting of all faces whose vertices lie in the set $X$.  
\end{definition}

\begin{theorem}[Hochster's Formula]\label{hochster}
Let $\C$ be a clutter without isolated vertices.  Then the multigraded Betti numbers of $I(\C)$ are given by
\[
\beta_{i-1, x^A}(I(\C)) = \dim_{{\bf k}} (\tilde{H}_{|A| - i - 1}(\Delta_\C[A]), \text{k}).
\]
In particular, $\pd(I(\C))$ is the least integer $i$ such that 
\[
\tilde{H}_{|A| - i - j - 1}(\Delta_\C[A]) = 0
\]
for all $j > 0$ and $A \subseteq V(\C)$.  
\end{theorem}

We can bound the homology of the complex $\Delta_\C$ through a specialization of Theorem \ref{hochster}: Setting $A = V(\C)$ yields the following corollary.

\begin{corollary}\label{homvanish}
If $\C$ is a clutter, then $\tilde{H}_k(\Delta_\C) = 0$ for $k < |V(\C)| - \pd(\C) -1$.  
\end{corollary} 

Using this corollary with Theorem \ref{edgewise} gives us the following.

\begin{corollary}
Let $\Sigma$ be a simplicial complex, and let $\C$ be the clutter of minimal non-faces of $\Sigma$.  Then 
\[
\tilde{H}_k(\Sigma) = 0
\]
whenever $k < \epsilon(\C) - 1$.  
\end{corollary}

\begin{proof}
Suppose $k < \epsilon(\C) - 1$.  By Theorem \ref{edgewise}, $\epsilon(\C) -1  \leq |V(\C)| - \pd(\C) - 1$, so the result follows from Corollary \ref{homvanish}. 
\end{proof}

\noindent \textbf{Acknowledgements}: The first author was partially supported by NSF grants DMS 0834050 and DMS 1104017.  We thank Sara Saeedi Madani for pointing out some errors in the first version of this paper.  We also thank Alexander Engstr\"om, Seyed Amin Seyed Fakhari, Tai Ha, Craig Huneke, Jason McCullough, Susan Morey, Adam Van Tuyl, and Russ Woodroofe for many helpful conversations and emails.

\end{document}